\documentclass[oneside,leqno,10pt]{article}
\usepackage{amssymb,amsmath,latexsym,amsthm,stmaryrd,amscd}

\usepackage{setspace}
\usepackage[nottoc]{tocbibind}

\def\ga{\mathfrak{a}}

\def\gg{\mathfrak{g}}

\def\gj{\mathfrak{j}}
\def\gk{\mathfrak{k}}
\def\gl{\mathfrak{l}}
\def\gm{\mathfrak{m}}

\def\go{\mathfrak{o}}

\def\gs{\mathfrak{s}}
\def\gt{\mathfrak{t}}
\def\gu{\mathfrak{u}}

\def\C{\mathbb{C}}

\def\R{\mathbb{R}}

\def\cD{\mathcal{D}}

\def\Span{{\rm Span}\,}
\def\Ad{{\rm Ad}}
\def\ad{{\rm ad}\,}

\def\rank{{\rm rank}\,}

\def\Aut{{\rm Aut}}
\def\Int{{\rm Int}}
\def\Out{{\rm Out}}

\def\Span{{\rm Span}\,}
\def\diag{{\rm diag}}

\newtheorem{theorem}[equation]{Theorem}

\newtheorem{lemma}[equation]{Lemma}

\newtheorem{corollary}[equation]{Corollary}

\newtheorem{proposition}[equation]{Proposition}

\def\sideremark#1{\ifvmode\leavevmode\fi\vadjust{\vbox to0pt{\vss
 \hbox to 0pt{\hskip\hsize\hskip1em
\vbox{\hsize2cm\tiny\raggedright\pretolerance10000 
 \noindent #1\hfill}\hss}\vbox to8pt{\vfil}\vss}}} 

\usepackage{tocloft}

\title{Homogeneity for a Class of Riemannian Quotient Manifolds}

\author{Joseph A. Wolf\,\footnote{Research partially supported by a 
grant from the Simons Foundation.\newline
}}

\date{\vskip -.5cm 16 September 2016}

\begin{document}

\maketitle

\abstract{We study riemannian coverings
$\varphi: \widetilde{M} \to \Gamma\backslash \widetilde{M}$ where
$\widetilde{M}$ is a normal homogeneous space $G/K_1$ fibered over
another normal homogeneous space $M = G/K$ and $K$ is locally
isomorphic to a nontrivial product $K_1\times K_2$\,.
The most familiar such fibrations
$\pi: \widetilde{M} \to M$ are the natural fibrations of Stieffel manifolds
$SO(n_1 + n_2)/SO(n_1)$ over Grassmann manifolds
$SO(n_1 + n_2)/[SO(n_1)\times SO(n_2)]$ and the twistor space bundles
over quaternionic symmetric spaces (= quaternion--Kaehler symmetric spaces 
= Wolf spaces).  The most familiar of these coverings
$\varphi: \widetilde{M} \to \Gamma\backslash \widetilde{M}$
are the universal riemannian coverings of spherical space forms.  When
$M = G/K$ is reasonably well understood, in particular when $G/K$ is a 
riemannian symmetric space or when $K$ is a 
connected subgroup of maximal rank in $G$, we show that the Homogeneity 
Conjecture holds for $\widetilde{M}$.  In other words we show that 
$\Gamma\backslash \widetilde{M}$ is
homogeneous if and only if every $\gamma \in \Gamma$ is an isometry of
constant displacement.  In order to find all the isometries of constant
displacement on $\widetilde{M}$ we work out the full isometry group of
$\widetilde{M}$, extending \' Elie Cartan's determination of the
full group of isometries of a riemannian symmetric space.  We also discuss
some pseudo--riemannian extensions of our results.
}

\begin{quote}
{\footnotesize 
\begin{spacing}{1.4}
\tableofcontents
\end{spacing}
}
\end{quote}

\section{Introduction}
\label{sec1}
\setcounter{equation}{0}

Some years ago I studied riemannian covering spaces $S \to \Gamma\backslash S$
where $S$ is homogeneous.  I conjectured that $\Gamma\backslash S$ is 
homogeneous if and only if every $\gamma \in \Gamma$ is an isometry of
constant displacement (now usually called {\em Clifford translations} or 
{\em Clifford--Wolf isometries}) on $S$.  I'll call that the {\em Homogeneity 
Conjecture}.
This paper proves the conjecture for a class of normal riemannian homogeneous 
spaces $\widetilde{M} = G/K_1$ that fiber over homogeneous spaces $M = G/K$
where $K_1$ is a local direct factor of $K$.  The principal examples are those 
for which $K$ is the fixed point set of an automorphism of $G$ and the Lie
algebra $\gk = \gk_1 \oplus \gk_2$ with $\dim\gk_1 \ne 0 \ne \gk_2$\,.  Those
include the cases where $G/K$ is an hermitian symmetric space, or a 
Grassmann manifold of rank $> 1$, or a
quaternion--Kaehler symmetric space (Wolf space), one of the
irreducible nearly--Kaehler manifolds of $F_4$, $E_6$, $E_7$ or $E_8$, or
everybody's favorite $5$--symmetric space $E_8/A_4A_4$\,.  See \cite{WG1968a}
and \cite{WG1968b} for a complete list.
\medskip

For lack of a better term
I'll refer to such spaces $\widetilde{M} = G/K_1$ as {\em isotropy--split} 
homogeneous spaces and to the fibration $\pi: \widetilde{M} \to M = G/K$ as an
{\em isotropy--splitting} fibration.
\medskip

Here we use isotropy--splitting fibrations $\pi: \widetilde{M} \to M$
as a bootstrap device to study riemannian coverings
$\varphi: \widetilde{M} \to \Gamma \backslash \widetilde{M}$.  Specifically,
$\pi$ is the projection given by $G/K_1 \to G/K$ with $K = K_1K_2$ where
$M$ and $\widetilde{M}$ are normal riemannian homogeneous spaces of the
of the same group $G$ and each $\dim K_i > 0$.  In particular,
$\pi: \widetilde{M} \to M$ is a principal $K_2$--bundle.   The point is to
choose the splitting of $K$ so that $M$ is reasonably well understood.  The 
most familiar example is the case where $\widetilde{M}$ is a Stieffel manifold
and $M$ is the corresponding Grassmann manifold.  More generally we study
the situation where 
\begin{equation}\label{setup}
\begin{aligned}
&G \text{ is a compact connected simply connected Lie group,}\\
&K = K_1K_2 \text{ where the $K_i$ are closed connected subgroups of $G$ 
	such that}\\
& \phantom{XXXX}\text{ (i)  $K = (K_1 \times K_2)/(K_1 \cap K_2)$,}
	\text{ (ii) } \gk_2 \perp \gk_1 \text{ and }
        \text{ (iii) } \dim \gk_1 \ne 0 \ne \dim \gk_2\,, \\
&\text{the centralizers } Z_G(K_1) = Z_{K_1}\widetilde{K_2}
	\text{ and } Z_G(K_2) = Z_{K_2}\widetilde{K_1} \text{ with }
	K_1 = \widetilde{K_1}^0 \text{ and } 
	K_2 = \widetilde{K_2}^0\,, \text{ and } \\
&M = G/K \text{ and } \widetilde{M} = G/K_1 \text{ are normal riemannian 
	homogeneous spaces of } G.
\end{aligned}
\end{equation}
Thus we may assume that the metrics on $M$ and $\widetilde{M}$ are the
normal riemannian metrics defined by the negative of the Killing form of $G$.
Note that $\widetilde{M}$ and $M$ are simply connected, because $G$ is
simply connected and $K_1$ and $K$ are connected.

\begin{lemma}\label{no-nonzero}
There is no nonzero $G$--invariant vector field on $M$.  In other
words, if $\gm = \gk^\perp$ then $\ad_\gg(\gk)|_\gm$ has no nonzero 
fixed vector.
\end{lemma}

\begin{proof} The centralizer $Z_G(K)$ is finite by (\ref{setup}).
\end{proof}

Lemma \ref{no-nonzero} is of course obvious whenever $\rank K = \rank G$,
in other words when the Euler characteristic $\chi(M) \ne 0$.  The point
here is that it holds as well when $\rank K < \rank G$.
\medskip

Important examples of $M$ include the irreducible riemannian symmetric spaces
$G/K$ with $K$ not simple, the
irreducible nearly--Kaehler manifolds of $F_4$, $E_6$, $E_7$ or $E_8$, and
the very interesting $5$--symmetric space $E_8/A_4A_4$\,.  We will
list these examples in detail and work out the precise structure of the
group $I(\widetilde{M})$ of all isometries of $\widetilde{M}$. That
is Theorem \ref{isogroup}, and Corollary \ref{CK-fields} identifies
all the Killing vector fields on $\widetilde{M}$ of constant length.
Killing vector fields of constant length are the infinitesimal version of
isometries of constant displacement.  After that we come to the main
result, Theorem \ref{cw}, which identifies all the isometries of
constant displacement on $\widetilde{M}$.  Applying it to a riemannian
covering $\widetilde{M} \to \Gamma \backslash \widetilde{M}$ we prove
the Homogeneity Conjecture for isotropy--split manifolds.  Then we sketch
the mathematical background and current state for the Homogeneity
Conjecture.
\medskip

In Section \ref{sec2} we view (compact) isotropy--splitting fibrations from
the viewpoint of the Borel--de Siebenthal classification (\cite{BdS1949},
or see \cite{W1966}) of pairs $(G,K)$
where $G$ is a compact connected simply connected simple Lie group and
$K$ is a maximal subgroup of equal rank in $G$.  This yields an explicit 
list.  We then run through the cases where $G/K$ is a compact irreducible
riemannian symmetric space with $\rank K < \rank G$; the only ones that yield 
isotropy--splitting fibrations are the fibrations of real Stieffel manifolds
over odd dimensional oriented real Grassmann manifolds.  These are examples 
with which one can calculate explicitly, and to which our principal results
apply.
\medskip

In Section \ref{sec3} we work out the full group of isometries of 
$\widetilde{M}$.  The method combines ideas from \' Elie Cartan's description
of the full isometry group of a riemannian symmetric space, Carolyn Gordon's
work on isometry groups of noncompact homogeneous spaces, and a theorem
of Silvio Reggiani.  The result
is Theorem \ref{isogroup}.  One consequence, Corollary \ref{CK-fields},
is a complete description of the Killing vector fields of constant
length on $\widetilde{M}$.
\medskip

Section \ref{sec4} is a digression in which we show that an 
appropriate form of Theorem \ref{isogroup}
holds in the equal rank case without the need for an isotropy--splitting
fibration.
\medskip

In Section \ref{sec5} we study isometries of constant displacement on 
$\widetilde{M}$ in the equal rank case, in other words when the 
Euler--Poincar\' e characteristic $\chi(M) \ne 0$. In that setting we
give a classification of homogeneous riemannian 
coverings $\widetilde{M} \to \Gamma \backslash \widetilde{M}$.
The arguments are modeled in part on those of the group manifold case
of riemannian coverings $S \to \Gamma\backslash S$ in \cite{W1963}.
The result is Theorem \ref{cw}, which is the principal result of this paper.
The main application is Corollary \ref{cw-cover}, which applies
Theorem \ref{cw} to riemannian coverings 
$\widetilde{M} \to \Gamma \backslash \widetilde{M}$.
\medskip

In Section \ref{sec6}  we study isometries of constant displacement on
$\widetilde{M}$ when $\chi(M) = 0$.  We work out a modification of the
proof of Theorem \ref{cw}, proving Theorem \ref{cw-less},  which 
characterizes the isometries of constant displacement on $\widetilde{M}$
for $\chi(M) = 0$.
\medskip

In Section \ref{sec7} we specialize Theorem \ref{cw} to the case where
$M$ is a compact irreducible riemannian symmetric space.  From the
classification and the isotropy--splitting requirement, the only cases
are the natural fibrations of Stieffel manifolds over odd dimensional oriented
real Grassmann manifolds.  There we characterize the isometries of constant
on $\widetilde{M}$ by a matrix calculation.
\medskip

In Section \ref{sec8} we apply our results on constant displacement
isometries to the Homogeneity Conjecture.  The main result of this paper,
Theorem \ref{conj-isotropy--split}, proves the conjecture for $\widetilde{M}$
when $\rank K = \rank G$, and also when $M$ is a riemannian symmetric space.
In particular it proves the conjecture for $\widetilde{M}$ when
$\widetilde{M} \to M$ is one of the fibrations described in Section \ref{sec2}.
We then describe the current state of the art for the Homogeneity Conjecture, 
its infinitesimal variation, and its extension to Finsler manifolds.  
Earlier work had proved it in many special cases, for example for
riemannian symmetric spaces, and its validity for isotropy--split manifolds 
extends our understanding of the area.
\medskip

In Section \ref{sec9} we show how our results on compact isotropy--split 
manifolds carry over (or, rather, often do not carry over) 
to the noncompact case.  There we see that 
$\widetilde{M}$ is pseudo--riemannian and we can't talk about isometries
of constant displacement.  Thus, in that setting, we concentrate on the
isometry group and on Killing
vector fields of constant length.  Of special interest here is the
case where the base $M$ of the isotropy--splitting fibration 
$\widetilde{M} \to M$ is a riemannian symmetric space of noncompact type,
but other cases of special interest are those for which the ``compact dual''
isotropy--splitting fibration $\widetilde{M}^u \to M^u$ has $3$--symmetric 
or $5$--symmetric base..  

\section{Some Special Classes of Isotropy--Splitting Fibrations}
\label{sec2}
\setcounter{equation}{0}
In this section we describe a number of interesting examples of
isotropy--splitting fibrations $\widetilde{M} \to M$.  Those are examples with 
which one can calculate explicitly, and to which our principal results
apply.
\medskip

Fix a compact connected simply connected Lie group $G$ and a maximal 
connected subgroup $K$ with $\rank K = \rank G$.  The 
Borel -- de Siebenthal classification of all
such pairs $(G,K)$ is in \cite{BdS1949}, or see \cite{W1966}.
\medskip

We recall that classification.  We may assume that $G$ is simple.
Fix a maximal torus $T \subset K$ 
of $G$ and a positive root system $\Delta^+(\gg_\C,\gt_\C)$.  Express the
maximal root $\beta = \sum_{\psi \in \Psi_G} n_\psi \psi$ where $\Psi_G$ is
the simple root system for $\Delta^+(\gg_\C,\gt_\C)$.  The coefficients
$n_\psi$ are positive integers, and the possibilities for $\gk$ correspond
to the simple roots $\psi_0$ for which either $n_{\psi_0} = 1$ or 
$n_{\psi_0} > 1$ with $n_{\psi_0}$ prime.  Fix one such, $\psi_0$\,, 
and write $n_0$ for $n_{\psi_0}$\,.
\medskip

If $n_0 = 1$ the simple root system
$\Psi_K = \Psi_G \setminus \{\psi_0\}$.  This is the case where $G/K$ is an
hermitian symmetric space.  If $n_0 > 1$ then
$\Psi_K = (\Psi_G \setminus \{\psi_0\})\cup\{-\beta\}$.  In this case either
$n_0 = 2$ and $G/K$ is a non--hermitian symmetric space, or $n_0 = 3$
and $G/K$ is a nearly--Kaehler manifold, or $n_0 = 5$ and
$G/K = E_8/A_4A_4$\,.
\medskip

If $\cD_G$ is the Dynkin diagram of $\gg$
then the diagram $\cD_K$ of $\gk$ is obtained as follows.  If 
$n_0 = 1$ then delete the vertex $\psi_0$ from $\cD_G$.  If $n_0 > 1$
the delete the vertex $\psi_0$ and adjoin the vertex $-\beta$.  
The simple root(s) not orthogonal to $\beta$, in other words the
attachment points for $-\beta$ to $\cD_G$\,, may or may not
disconnect $\cD_G$\,.  If there is disconnection then $K$ splits into
the form $K_1K_2$ of interest to us.

\subsection{Hermitian Symmetric Space Base}\label{ssec2a}
If $n_0 = 1$ then $K = SK'$ where $S$ is a circle group and $K' = [K,K]$
is semisimple.  The corresponding fibrations are
$$
\begin{aligned}
&G/K' \to G/K \text{ circle bundle over a compact hermitian 
	symmetric space and }\\
&G/S \to G/K \text{ principal $K'$--bundle over a compact hermitian 
        symmetric space.}
\end{aligned}
$$
In addition, if $\gg = \gs\gu(s+t)$ we can have 
$\gk' = \gs\gu(s)\oplus\gs\gu(t)$,  leading to fibrations 
$$
\begin{aligned}
&SU(s+t)/SU(s) \to SU(s+t)/S(U(s)U(t)) \text{ and  }
	SU(s+t)/U(s) \to SU(s+t)/S(U(s)U(t)),\\
&SU(s+t)/SU(t) \to SU(s+t)/S(U(s)U(t)) \text{ and  }
	SU(s+t)/U(t) \to SU(s+t)/S(U(s)U(t)).
\end{aligned}
$$
\subsection{Quaternion--Kaehler Symmetric Space Base} \label{ssec2b}

If $n_0 = 2$ then $K$ is simple except in the cases
$$
\begin{aligned}
& G/K = SO(s+t)/SO(s)SO(t) \text{ with } 2 < s \leqq t
	\text{ and } st \text{ even, }\\
& G/K = Sp(s+t)/Sp(s)Sp(t) \text{ with } 1 \leqq s \leqq t, \\
& G/K = G_2/A_1A_1\,, F_4/A_1C_3\,, E_6/A_1A_5\,, E_7/A_1D_6 \text{ or }
	E_8/A_1E_7\,.
\end{aligned}
$$
In the $SO$ cases, $G/K$ is a quaternion---Kaehler symmetric space for $s = 3$ 
and for $s = 4$.  In the $Sp$ cases $G/K$ is a quaternion--Kaehler symmetric 
space for $s =1$.  In the exceptional group cases $G/K$ always is a 
quaternion--Kaehler symmetric space.

\subsection{Nearly-Kaehler $3$--Symmetric Space Base}\label{ssec2c}
If $n_0 = 3$ then either $K$ is simple and $G/K = G_2/A_2$ or $E_8/A_8$\,, 
or $K$ is not simple and $G/K$ is one of the nearly--Kaehler manifolds
$F_4/A_2A_2$\,, $E_6/A_2A_2A_2$\,, $E_7/A_2A_5$ or $E_8/A_2E_6$\,.
In the $F_4$ case one of the $A_2$ is given by long roots and the other is
given by short roots.  In each case we have $\gk = \ga_2 \oplus \gk''$ where
the $\ga_2$ is given by long roots.  The $3$--symmetry on $G/K$ is given
by one of the central elements of $\exp(\ga_2) = SU(3)$.  It defines the
almost--complex structure on $G/K$, which satisfies the nearly--Kaehler
condition.  The corresponding fibrations are
$$
G/K'' \to G/K \text{ principal } SU(3)\text{--bundle and }
G/SU(3) \to G/K \text{ principal } K''\text{--bundle.}
$$

\subsection{$5$--Symmetric Space Base}\label{ssec2d}
If $n_0 = 5$ then $G/K = E_8/A_4A_4$\,, where the first $A_4$ acts on the 
complxified tangent space by a sum of $5$ dimensional representations  and
the second $A_4$ acts by a sum of $10$ dimensional representations.  This
leads to two different principal $SU(5)$--bundles 
$E_8/SU(5) \to E_8/SU(5)SU(5)$.

\subsection{Odd Real Grassmann Manifold Base}\label{ssec2e}
The Borel -- de Sibenthal classification, just described, gives
the classification of irreducible compact riemannian symmetric spaces $S$
with Euler characteristic $\chi(S) \ne 0$.  There are other symmetric
spaces to which our results will apply, corresponding to the isotropy--split
fibrations $\pi: \widetilde{M} \to M$ where the base $M$ is an irreducible 
compact riemannian symmetric space $G/K$ such that $\rank G > \rank K$.
According to the classification of symmetric spaces, the only such $G/K$ are
{\footnotesize
$$
SU(n)/SO(n), SU(2n)/Sp(n),
SO(2s+2t+2)/[SO(2s + 1)\times SO(2t+1)], E_6/F_4, E_6/Sp(4), (K\times K)/diag(K).
$$
}
\hskip -4pt Note that $SU(4)/SO(4) = SO(6)/[SO(3) \times SO(3)]$.  Thus the only such
symmetric spaces $G/K$ that satisfy (\ref{setup}) are the oriented
real Grassmann manifolds $SO(2s+2t+2)/[SO(2s + 1)\times SO(2t+1)]$.  Thus
the corresponding fibrations are
$$
\begin{aligned}
\pi: \widetilde{M} \to M &\text{ given by } G/K_1 \to G/K_1K_2 \text{ where }\\
     &G = SO(2s+2t+2)/SO(2s+1), K_1 = SO(2s + 1) \text{ and } K_2 =  SO(2t+1).
\end{aligned}
$$
The odd spheres are completely understood (\cite{W1961} and \cite{W1966}),
and in any case they do not lead to isotropy--split fibrations, so
we put those cases aside and assume $s, t > 0$.

\section{The Isometry Group of $\widetilde{M}$}
\label{sec3}
\setcounter{equation}{0}
We look at an isotropy--splitting  fibration $\pi: \widetilde{M} \to M$, 
given by $G/K_1 \to G/K$ in (\ref{setup}).  As noted there we 
assume that the metrics on $M$ and $\widetilde{M}$ are the
normal riemannian metrics defined by the negative of the Killing form of $G$.  
Now we work out the isometry groups $I(\widetilde{M})$.  

\begin{lemma} The right action of $\widetilde{K_2}$ on $\widetilde{M}$, given by
$r(k_2)(gK_1) = gK_1 k_2^{-1} = gk_2^{-1}K_1$\,, is a well defined
action by isometries.  The fiber of $\pi: \widetilde{M} \to M$
through $gK_1$ is $r(K_2)(gK_1)$.
\end{lemma}
\vskip -10pt
\begin{equation}
\text{Let } F \text{ denote the fiber } r(K_2)(1K_1) \text{ of }
\pi: \widetilde{M} \to M, \text{ so } gF \text{ is the fiber }
\pi^{-1}(gK).
\end{equation}
We have larger (than $G$) transitive
groups of isometries of $\widetilde{M}$ given by
\begin{equation}\label{g-tilde}
\widetilde{G} = G \times r(\widetilde{K_2}) \text{ and }
\widetilde{G}^0 = G \times r(K_2) \text{ acting by }
        (g,r(k_2)):xK_1 \mapsto g(xK_1)k_2^{-1} = gxk_2^{-1}K_1\,.
\end{equation}
Every $\widetilde{g} = (g,r(k_2)) \in \widetilde{G}$ sends fiber to fiber
in $\widetilde{M} \to M$ and induces the isometry $g: M \to M$ of $M$.
\medskip

Specializing a theorem of Reggiani \cite[Corollary 1.3]{R2010} we have

\begin{theorem}\label{ident-comp}
Suppose that the riemannian manifold $\widetilde{M} = G/K_1$ is irreducible.
Then $\widetilde{G}^0$ is the identity component $I^0(\widetilde{M})$ of 
its isometry group.
\end{theorem}

\begin{corollary}\label{CK-fields}
Suppose that the riemannian manifold $\widetilde{M} = G/K_1$ is irreducible.
\begin{itemize}
\item[{\rm (1)}] The algebra of all Killing vector fields on $\widetilde{M}$ is
$\widetilde{\gg} = \gg \oplus dr(\gk_2)$.  
\item[{\rm (2)}] The set of all constant length 
Killing vector fields on $\widetilde{M}$ is $\gl \oplus dr(\gk_2)$ where \\
\phantom{XXXXXXXXXX}
$\gl = \{\xi \in \gg \mid \xi \text{ defines a constant length Killing
vector field on } \widetilde{M}\}.$  
\item[{\rm (3)}] 
$\gl = \{\xi \in \gg \mid \xi \text{ defines a constant length Killing
vector field on } M \}.$  
\item[{\rm (4)}]  If $\rank K = \rank G$ then $\gl = 0$, 
so the set of all constant length Killing vector fields on $\widetilde{M}$ is
$dr(\gk_2)$.  That applies in particular to the special classes of 
Sections \ref{ssec2a} through \ref{ssec2d}.
\end{itemize}
\end{corollary}

\begin{proof} The first assertion is immediate from Theorem \ref{ident-comp}.
For the second assertion, $dr(\gk_2)$ consists of Killing vector fields of
constant length on $\widetilde{M}$ because every $\xi \in dr(\gk_2)$ is
centralized by the transitive isometry group $G$.  
\smallskip

For the third assertion, let $\xi$ be a Killing vector field of 
constant length on $\widetilde{M}$.  Using Theorem \ref{ident-comp}
express $\xi = \xi' + \xi''$ where $\xi' \in \gg$ and $\xi'' \in dr(\gk_2)$.
The fibers of
$\pi: \widetilde{M} \to M$ are just the orbits of $r(K_2)$ and are group
manifolds, so $\xi''$ is a Killing vector field of constant length on
$\widetilde{M}$.  Further,
$\xi'\perp \xi''$ at every point of $\widetilde{M}$.  Now $\xi'$
is a Killing vector field of constant length on $\widetilde{M}$.  It
follows that $\xi'$ is a Killing vector field of constant length on $M$
as well.  
\smallskip

For the fourth assertion, let $\rank K = \rank G$, so the Euler--Poincar\' e 
characteristic $\chi(M) > 0$.  Then the vector field $\xi'$ (of the
argument for (3) just above) must have a zero on $M$.  Thus
$\xi' = 0$ and $\xi = \xi'' \in dr(\gk_2)$.
\end{proof}

\begin{corollary}\label{ck2}
Every isometry of $\widetilde{M}$ normalizes $r(K_2)$ and thus sends fiber 
to fiber in $\pi:\widetilde{M} \to M$.
\end{corollary}

Now we start to extend this to a structure theorem for the full isometry group
$I(\widetilde{M})$ under the constraint of (\ref{setup}).  

The normalizer of $K_1$ in $G$ also normalizes the centralizer of $K_1$,
thus normalizes $K_2$ and thus normalizes $K$.  That shows
\begin{lemma}\label{normalizer}
The normalizer of $K_1$ in $G$ is contained in the normalizer of $K$ in $G$.
\end{lemma}

Now we follow the basic idea of \' E. Cartan's determination of the 
holonomy group and then the isometry group of a riemannian 
symmetric space (\cite{C1927a}, \cite{C1927b}; or see \cite{W1966}).
Write $\Out(G)$ for the quotient $\Aut(G)/\Int(G)$ of the automorphism 
group by the normal subgroup of inner automorphism, and similarly
$\Out(K_1) = \Aut(K_1)/\Int(K_1)$.  We also need the relative group
\begin{equation}\label{outer}
\Out(G,K_1) = \bigl \{\alpha \in \Aut(G) \mid \alpha(K_1) = K_1\}/
		\{\alpha \in \Int(G) \mid \alpha(K_1) = K_1\}
		\subset \Out(G,K).
\end{equation}
The inclusion in (\ref{outer}) follows because $K_1$ is a local
direct factor of $K$.
In many cases $\Out(G,K_1) = \Out(G,K)$ because $\gk_2$ is the 
$\gg$--centralizer of $\gk_1$ and $\gk_1 \not \cong
\gk_2$.  But there are exceptions, such as orthocomplementation (which
exchanges the two factors of $K$) in the cases of Stieffel manifold fibrations 
$$
\begin{aligned}
&SO(2k)/SO(k) \to SO(2k)/[SO(k)\times SO(k)], \\
&SU(2k)/U(k) \to SU(2k)/S(U(k) \times U(k)) \text{ and } \\
&Sp(2k)/Sp(k) \to Sp(2k)/[Sp(k)\times Sp(k)].
\end{aligned}
$$ 
There are other exceptions, including $E_6/[A_2A_2A_2]$, but neither 
$F_4/A_2A_2$ nor $E_8/A_4A_4$ is an exception.

\begin{lemma}\label{inn-out}
Suppose that $\rank K = \rank G$.  Let $\alpha \in \Aut(G)$ preserve $K_1$
$($and thus also $K_2$ so $\alpha(K) = K)$.  Then the following conditions
are equivalent: 
{\rm (i)} $\alpha|_K$ is an inner automorphism of $K$, 
{\rm (ii)} as an isometry, $\alpha \in I^0(M)$, and 
{\rm (iii)} as an isometry, $\alpha \in I^0(\widetilde{M})$.
\end{lemma}

\begin{proof} Suppose that $\alpha|_K$ is an inner automorphism.  Then 
we have $k_0 \in K$ such that
$\alpha(k) = k_0kk_0^{-1}$ for every $k \in K$.  Thus
$\alpha' := \Ad(k_0^{-1})\cdot \alpha$ is an isometry of $M$ that
belongs to the same component of $I(M)$ as $\alpha$.  Let $T$ be a maximal
torus of $K$ that contains $k_0$\,.  Then $\alpha' := \Ad(k_0^{-1})\cdot \alpha$
is an isometry of $M$ that belongs to the same component $I^0(M)\alpha$
as $\alpha$.  Now $\alpha'(t) = t$ for every $t \in T$ so there is an
element $t_0 \in T$ such that $\alpha'(g) = t_0gt_0^{-1}$ for every 
$g \in G$.  Consequently $\alpha'' := \Ad(t_0^{-1})\cdot \alpha'$
is the identity in $I(M)$ and belongs to the same component of $I(M)$ as
$\alpha$.  It follows that $\alpha''$ is the identity in $I(\widetilde{M})$
and belongs to the same component of $I(\widetilde{M})$ as $\alpha$.  Thus,
as an isometry, $\alpha \in I^0(M)$ and $\alpha \in I^0(\widetilde{M})$.
We have shown that (i) implies (ii) and (iii).
\medskip

Suppose that $\alpha|_K$ is an outer automorphism.  Then $\alpha \in I(M)$
represents a non-identity
component of the isotropy subgroup at $1K$, i.e. $\alpha \notin I^0(M)$.
In view of Corollary \ref{ck2} we have a natural continuous homorphism
of $I(\widetilde{M})$ to $I(M)$ that maps $I^0(\widetilde{M})$ onto
$I^0(M)$.  Thus $\alpha \notin I^0(\widetilde{M})$.  We have shown that
if (i) fails then (ii) and (iii) fail.  Thus (ii) implies (i) and (iii)
implies (i).  That completes the proof.
\end{proof}

We reformulate Lemma \ref{inn-out} as

\begin{lemma}\label{isotropy}
Let $\rank K = \rank G$.
Let $\widetilde{H}$ denote the isotropy subgroup of $I(\widetilde{M})$ at the
base point $\widetilde{x_0} = 1K_1$\,.  Then the identity component
$\widetilde{H}^0$ is 
$K_1\cdot\{(k_2,r(k_2)) \in \widetilde{G} \mid k_2 \in K_2\}$ and
$$
\widetilde{H} = {\bigcup}_{\alpha \in \Out(G,K_1)\,, \beta \in \Out(G,K_2)}\,  
K_1\alpha\cdot
\bigl \{(k_2,r(k_2)) \in \widetilde{G} \mid 
k_2 \in \widetilde{K_2}\beta\bigr \}
$$
Given $\alpha, \alpha' \in \Out(G,K_1)$ and $\beta, \beta' \in \Out(G,K_2)$,
the components $K_1\alpha\cdot
\bigl \{(k_2,r(k_2)) \in \widetilde{G} \mid
k_2 \in \widetilde{K_2}\beta\bigr \} = K_1\alpha\cdot
\bigl \{(k_2,r(k_2)) \in \widetilde{G} \mid
k_2 \in \widetilde{K_2}\beta\bigr \}$ if and only if both $\alpha = \alpha'$
and $\beta = \beta'$ modulo inner automorphisms.
\end{lemma}
\begin{proof} The fiber $F = r(K_2)\widetilde{x_0}$ is the group manifold
$K/K_1$\,, and $\widetilde{H}$ preserves $F$ by Corollary \ref{CK-fields}.  The 
isotropy subgroup of $I(F)$ at $\widetilde{x_0}$ has identity component
$diag(K_2) = \{(k_2,r(k_2)) \in \widetilde{G} \mid k_2 \in K_2\}$. 
The group $diag(K_2)$ is connected and is contained in $\widetilde{H}$ 
because it leaves $\widetilde{x_0}$ fixed, so 
$diag(K_2) \subset \widetilde{H}^0$.  Also $K_1 = G\cap \widetilde{H}
\subset \widetilde{H}^0$\,.  It follows that 
$\widetilde{H}^0 = K_1\cdot diag(K_2)$\,, as asserted.
\smallskip

The inclusion ${\bigcup}_{\alpha \in \Out(G,K_1)\,, \beta \in \Out(G,K_2)}\,
K_1\alpha\cdot
\bigl \{(k_2,r(k_2)) \in \widetilde{G} \mid
k_2 \in \widetilde{K_2}\beta\bigr \} \subset \widetilde{H}$ is clear.
\smallskip

Now let $h \in \widetilde{H}$.  Then $h(F) = F$ by Corollary \ref{ck2},
so conjugation by $h$ gives an automorphism $\beta$ of 
$diag(K_2) = \{(k_2,r(k_2)) \in \widetilde{G} \mid k_2 \in K_2\}$, and we
view $\beta$ as an element of $\Out(G,K_2)$.  Furthermore, conjugation by $h$
gives an automorphism $\alpha$ of $G$ and we view $\alpha$ as an element
of $\Out(G,K_1)$.  Thus $\widetilde{H}$ is contained in
${\bigcup}_{\alpha \in \Out(G,K_1)\,, \beta \in \Out(G,K_2)}\,
K_1\alpha\cdot
\bigl \{(k_2,r(k_2)) \in \widetilde{G} \mid
k_2 \in \widetilde{K_2}\beta\bigr \}$, and they are equal, as asserted.
\smallskip

Finally, the last statement is immediate from Lemma \ref{inn-out}.
\end{proof}

We now define two subgroups of isometry groups by
\begin{equation}\label{daggers}
G^\dagger = {\bigcup}_{\alpha \in \Out(G,K_1)}\, G \alpha \subset I(M) 
\text{ and } \widetilde{G}^\dagger = 
{\bigcup}_{\alpha \in \Out(G,K_1)\,, \beta \in \Out(G,K_2)}\, 
G\alpha \cdot r(K_2)\beta \subset I(\widetilde{M}).
\end{equation}
Here $g\alpha$ acts on $M$ by $xK \mapsto g\alpha(x)K$ and on 
$\widetilde{M}$ by  $xK_1 \mapsto g\alpha(x)K_1$\,, and $r(k_2)\beta$
acts on $\widetilde{M}$ by $xK_1 \mapsto x\beta(k_2)^{-1}K_1$\,.
\medskip

\begin{theorem}\label{isogroup}
Let $\pi: \widetilde{M} \to M$ be an isotropy--split fibration as in
{\rm (\ref{setup})}.  Suppose that $\rank K = \rank G$.  Then
the identity component $I^0(\widetilde{M}) = \widetilde{G}^0$ and 
the full isometry groups $I(\widetilde{M}) = \widetilde{G}^\dagger$.
\end{theorem}

\begin{proof} The first statement repeats
Theorem \ref{ident-comp}.  As $G$ is transitive on
$\widetilde{M}$ one has $I(\widetilde{M}) = G\widetilde{H}$, and
the assertion follows from Lemma \ref{isotropy}.
\end{proof}

\section{Digression: The Isometry Group Without a Splitting Fibration}
\label{sec4}
\setcounter{equation}{0}

Theorem \ref{isogroup} holds without the splitting fibration.  
That result is useful and we indicate it here.
\medskip

A ``degenerate'' form of Lemma \ref{inn-out} holds as follows: Let $A$ be a
compact connected semisimple Lie group and $B$ a closed connected subgroup
of maximal rank.  Let $N = A/B$, coset space with the normal riemannian metric
from the negative of the Killing form of $A$.  Let $\alpha$ be an automorphism
of $A$ that preserves $B$.  Then $\alpha|_B$ is an inner automorphism of $B$
if and only if, as an isometry, $\alpha \in I^0(N)$.  The proof is immediate
from the proof of Lemma \ref{inn-out}.
\medskip

Next, a ``degenerate'' form of Lemma \ref{isotropy}
holds as follows: Let $H$ denote the isotropy subgroup of $I(N)$ at the 
base point $1K$.  Then $H = {\bigcup}_{\alpha \in \Out(A,B)}\, H^0\alpha$.
Given $\alpha, \alpha' \in \Out(A,B)$ the components $H^0\alpha =
H^0\alpha'$ if and only if $\alpha = \alpha'$ modulo inner automorphisms.
The argument follows by specializing the proof of Lemma \ref{isotropy}
\medskip

Finally, a ``degenerate'' form
of Theorem \ref{isogroup} holds as follows: Let $A$ be a
compact connected semisimple Lie group and $B$ a closed connected subgroup
of maximal rank.  Let $N = A/B$, coset space with the normal riemannian metric
from the negative of the Killing form of $A$,  Then $I^0(N)$ is given by
Theorem \ref{ident-comp} and, in view of the remarks just above,
$I(N) = {\bigcup}_{\alpha \in \Out(A,B)}\, A\alpha$.
\medskip

We summarize these comments as
\begin{theorem}\label{no-fibration}
Let $A$ be a
compact connected semisimple Lie group and $B$ a closed connected subgroup
of maximal rank.  Let $N = A/B$, coset space with the normal riemannian metric
from the negative of the Killing form of $A$.  Then $I^0(N)$ is given by
{\rm Theorem \ref{ident-comp}} and 
$I(N) = {\bigcup}_{\alpha \in \Out(A,B)}\, A\alpha$.
\end{theorem}

\section{Isometries of Constant Displacement: Case $\chi(M) \ne 0$}
\label{sec5}
\setcounter{equation}{0}
Fix an isotropy--splitting fibration $\pi: \widetilde{M} \to M$ as in
(\ref{setup}).  In this section we look at isometries of constant
displacement on $\widetilde{M} = G/K_1$ where the Euler--Poincar\' e
characteristic $\chi(M) \ne 0$, in other words where $\rank K = \rank G$.
Then $\chi(M) = |W_G|/|W_K| > 0$ where $W$ denotes the Weyl group.
Some important examples are the isotropy--splitting fibrations
described in Sections
\ref{ssec2a}, \ref{ssec2b}, \ref{ssec2c} and \ref{ssec2d}.
\medskip

In Section \ref{sec6} we will look at cases where $\chi(M) = 0$, and in 
Section \ref{sec7} we will consider the remaining cases where $M$ is an 
irreducible riemannian symmetric space.

\begin{lemma} \label{fp-below}
If $\rank K = \rank G$ and $\widetilde{g} = (g,r(k_2)) \in \widetilde{G}$
then there is a fiber
$xF = \pi^{-1}(xK)$ of $\widetilde{M} \to M$ that is invariant
under the action of $\widetilde{g}$ on $\widetilde{M}$.
\end{lemma}
\begin{proof} Every element of the compact connected Lie group 
$\widetilde{G}$ belongs to a maximal torus, thus is conjugate to 
an element of $K \times r(K_2)$, and consequently has a fixed point on $M$.
\end{proof}

We need an observation concerning the geodesics in $\widetilde{M}$ and $F$.  
\begin{lemma}\label{go-space} The isotropy--split manifold
$\widetilde{M}$ is a geodesic orbit space, i.e. every geodesic is the orbit 
of a one--parameter subgroup of $G$\,.  The fiber $F$ of
$\widetilde{M} \to M$ is totally geodesic in $\widetilde{M}$ and also 
is a geodesic orbit space.  Every geodesic of $\widetilde{M}$ tangent
to $F$ is of the form 
$t \mapsto \exp(t\xi)x$ with $x \in F$ and $\xi \in \gk_2$\,.
\end{lemma}
\begin{proof} Recall that $\widetilde{M}$ is a normal homogeneous space 
relative to the group $G$ and the riemannian metric given by the negative 
of the Killing form $\kappa$ of $G$.  Write $\gg = \gk_1 + \gm_1$ where 
$\gm_1 = \gk_1^\perp$ relative to $\kappa$.  Write 
$\langle \,\cdot\, , \,\cdot\,\rangle$ for $-\kappa$.
It is positive definite on $\gg$.  If $\xi, \eta, \zeta \in \gm_1$ then
$\ad(\xi)$ is antisymmetric relative to $\langle \,\cdot\, , \,\cdot\,\rangle$
so $0 = \langle [\xi,\eta],\zeta\rangle + \langle \eta,[\xi,\zeta]\rangle
= \langle [\xi,\eta]_{\gm_1},\zeta\rangle + \langle \eta,[\xi,\zeta]_{\gm_1}
\rangle$.  In other words (see \cite[Definition 1.3]{G1996}),
\begin{equation}\label{nat-red}
\text{the } G \text{--homogeneous space } \widetilde{M} \text{ is naturally
reductive relative to } G \text{ and } \gg = \gk_1 + \gm_1\,.
\end{equation}
If $\xi \in \gm_1$ now (\cite{KV1991} or see \cite{G1996}) 
$t \mapsto \exp(t\xi)\cdot 1K_1$ is a geodesic in $\widetilde{M}$.  In 
particular $\widetilde{M}$ is a geodesic orbit space and $F$ is totally
geodesic in $\widetilde{M}$.  But $F$ is a riemannian symmetric space
under $K_2 \times r(K_2)$ with the metric obtained by restriction of
$\langle \,\cdot\, , \,\cdot\,\rangle$.  Thus
every geodesic of $\widetilde{M}$ tangent to $F$ at $1K_1$ has form 
$t \mapsto \exp(t\xi)(1K_1)$
with $\xi \in \gk_2$\,.  As $K_2$ acts transitively on $F$ with finite
kernel every geodesic in $F$ has form $t \mapsto \exp(t\xi)x$
with $x \in F$ and $\xi \in \gk_2$\,.
\end{proof}

Our principal results, starting with Proposition \ref{cw-conn} just below,
will depend on a certain flat rectangle argument.  The idea is that we have
two commuting Killing vector fields $\xi_1$ and $\xi_2$\,, typically
$\xi_2 \in dr(\gk_2)$ and $\xi_1 \in \gg$, such that $\xi_1 \perp \gk_1$ and 
both $g\times r(k_2) = \exp(\xi_1 + \xi_2)$ and $r(k_2) = \exp(\xi_2)$
have the same constant displacement.  Then the 
$\exp(t_1\xi_1 + t_2\xi_2)(1K_1)$, for $0 \leqq t_i \leqq 1$, 
form a flat rectangle.  There $r(k_2)$ is displacement along one side while
$g\times r(k_2)$ is displacement along the diagonal.  Since these displacements
are the same we argue that $\xi_1 = 0$.
\medskip

\begin{proposition}\label{cw-conn}
Suppose that $\rank K = \rank G$.
Let $\Gamma$ be a subgroup of $\widetilde{G}$ such that every
$\gamma \in \Gamma$ is an isometry of constant displacement on $\widetilde{M}$.
Then $\Gamma \subset (Z_G \times r(K_2))$ where $Z_G$ denotes the center
of $G$.
\end{proposition}

\noindent {\em Note}: Proposition \ref{cw-conn} applies in particular to the 
isotropy--split fibrations $\widetilde{M} \to M$ described in Sections
\ref{ssec2a} through \ref{ssec2d}.

\begin{proof} Let $\gamma = (g,r(k_2)) \in \Gamma$.  By Lemma \ref{fp-below} 
and conjugacy of maximal tori in $G$, we
have $h \in G$ such that $\gamma(hF) = hF$.  Since both $Z_G \times r(K_2)$ and
``constant displacement'' are fixed under $\Ad(G)$ we may replace
$\gamma$ by its $\Ad(G)$--conjugate $(h^{-1},1)(g,r(k_2)) (h,1)$, which 
preserves $F$ and still consists of isometries of constant displacement.
That done, $\gamma \in (K_1K_2 \times r(K_2))$.
\medskip

The group $K_1$ fixes the base point $\widetilde{x_0} = 1K_1 \in \widetilde{M}$.
If $k \in K_2$ then $K_1k\widetilde{x} = kK_1\widetilde{x} 
= k\widetilde{x}$.  Now $K_1$ fixes every point of $F$, so
$\gamma|_F \in (K_2 \times r(K_2))$.
As $F$ is totally geodesic in $\widetilde{M}$,\, $\gamma|_F$ is an isometry
of constant displacement on $F$.  Now \cite[Theorem 4.5.1]{W1962} says that 
either $\gamma|_F \in (K_2 \times r(\{1\}))$ or $\gamma|_F \in r(K_2)$.
\smallskip

Suppose that $\gamma|_F = zk_2 \in Z_GK_2$\,.  Then $\gamma = zk_1k_2$ 
and also has constant 
displacement on $F$, hence on $\widetilde{M}$.  Let $T_i \subset K_i$ be a 
maximal torus, so $T := T_1T_2$ is a maximal torus of $K$, and thus of $G$. 
Replace $\gamma$ by a conjugate and assume $\gamma = zk_2 \in Z_GT_2$\,.
Lemma \ref{go-space} gives us $\xi \in \gt_2$ such that $\exp(t\xi)\cdot1K_1\,,
0 \leqq t \leqq 1$, is the minimizing geodesic in $\widetilde{M}$ from
$1K_1$ to $zk_2K_1$\,.  In particular the (constant) displacement of
$\gamma = zk_2$ is $||\xi||$.  Let $w$ belong to the Weyl group $W(G,T)$,
say $w = \Ad(s)|_\gt$ where $s$ normalizes $T$.  Then $w(\gamma) = szk_2s^{-1}$
has the same constant displacement $||\xi||$ as does $\gamma$.  Note that
$w(\gamma)\cdot 1K_1 = \exp(w(\xi))\cdot 1K_1$\,.  Decompose $w(\xi) =
w(\xi)' + w(\xi)''$ where $w(\xi)' \in \gt_1$ and $w(\xi)'' \in \gt_2$\,.
Then $\exp(w(\xi))\cdot 1K_1 = \exp(w(\xi)'')\cdot 1K_1$ so
$||w(\xi)''|| = ||w(\xi)||$.  This says $w(\xi) \in \gt_2$ for every
$w \in W(G,T)$.  But $W(G,T)$ acts irreducibly on $\gt$, so an orbit $\ne 0$
cannot be confined to a proper subspace.  This contradicts
$\gamma = k_2 \in K_2$\,.  We conclude $\gamma \in (Z_G\times r(K_2))$.
\medskip

We have just shown that every $\gamma \in \Gamma$ is $\Ad(G)$--conjugate to
an element of $Z_G\times r(K_2)$.  As $G$ centralizes both $Z_G$ and 
$r(K_2)$ it follows that $\Gamma \subset (Z_G \times r(K_2))$.
\end{proof}

Proposition \ref{cw-conn} holds whether or not the maximal rank
subgroup $K$ of $G$ is a maximal subgroup, describing the.
groups of isometries of constant displacement on $\widetilde{M}$ that 
are contained in the identity component $I^0(\widetilde{M})$.  
Next, we look in the other components of $I(\widetilde{M})$.  That
will require an understanding of the full isometry group $I(M)$.

\begin{lemma}\label{outers}
Suppose that $\rank K = \rank G$.
Let $\alpha \in \Out(G,K_1)$ and $\gamma \in \widetilde{G}\alpha$ such that
both $\gamma$ and $\gamma^2$ are isometries of constant displacement on
$\widetilde{M}$.  Then $\alpha|_{K_1}$ is an inner automorphism of $K_1$
and $\gamma \in (Z_G \times r(K_2))$.
\end{lemma}

\begin{proof}  Let $\gamma =  (\alpha \cdot g) \times (r(k_2)\cdot\beta)$ 
as in (\ref{daggers}) and Theorem \ref{isogroup}, using
$\alpha \cdot g = \alpha(g)\cdot \alpha$.  Exactly as in the proof of 
Proposition \ref{cw-conn} we may assume that 
$(g \times r(k_2)\cdot \beta)F = F$.  Now
$gF = F$ and 
$\gamma(F) = (\alpha \cdot g)(F) = 
\alpha(F)$.  But $\alpha K_1 = K_1$ and (\ref{setup}) together imply
$\alpha(K_2) = K_2$.  Thus $\alpha(K) = K$, in other words $1K$ is a fixed
point for $\alpha$ on $G/K$; equivalently, $\alpha(F) = F$.  Now 
$\gamma(F) = F$.  As $F$ is totally geodesic in $\widetilde{M}$,
$\gamma|_F$ has constant displacement on $F$, so $\gamma|_F \in r(K_2)$.
In particular $\beta = 1$ and 
$\gamma \in ((\alpha \cdot Z_G K_1) \times r(K_2))$.
\medskip

Now we argue along the lines of the proof of Proposition
\ref{cw-conn}.  Both $\gamma$ and $r(k_2)$ have the same constant displacement
(call it $c$) on $F$, thus on $\widetilde{M}$.  
Following de Siebenthal \cite{S1956} we have an $\alpha$--invariant maximal 
torus $T_1 \subset K_1$ such that (after a $K_1$--conjugation) 
$\alpha k_1 \in \alpha T_1^\alpha$ where $T_1^\alpha$ is the fixed point
set of $\alpha$ on $T_1$\,.  Express 
$\alpha k_1 = \alpha \cdot \exp(\xi_1)$
where $\xi_1 \in \gt_1^\alpha$.  Let $T_2$ be a maximal torus of $K_2$
such that $k_2 = \exp(\xi_2)$ for some $\xi_2 \in \gt_2$\,.  
Let $\xi = \xi_1 + \xi_2$\,.  We may assume the $\xi_i$ chosen so that
$\alpha \cdot \exp(t\xi)\cdot 1K_1$\,, $0 \leqq t \leqq 1$, is a minimizing 
geodesic from $1K_1$ to $\gamma(1K_1)$.  Then $\exp(t\xi_2)\cdot 1K_1$\,, 
$0 \leqq t \leqq 1$, also is a minimizing geodesic from $1K_1$ to 
$\gamma(1K_1)$.  Now the corresponding vector fields $d\alpha(\eta)$ and 
$\eta_2$ on $\widetilde{M}$ satisfy $||d\alpha(\eta)|| = c = ||\eta_2||$ 
at every point of $\widetilde{M}$.  
\medskip

If $\xi_1 \ne 0$ then, as we move a little bit away from $1K_1$ in some 
direction orthogonal to $F$, $||d\alpha(\eta_1)||$ increases from $0$.  
That increase in $||d\alpha(\eta_1)||$ would cause an increase in 
$||d\alpha(\eta)||$ because $d\alpha(\eta_1)$ and $d\alpha(\eta_2)$ would
remain close to orthogonal.  
We conclude $\xi_1 = 0$.  Now $\gamma = \alpha \times r(k_2)$.  
Again, if $\alpha \ne 1$ then, as we move away from $1K_1$
in some direction, the displacement of $\alpha$ would increase from
$0$, and that would cause an increase in the displacement
of $\gamma$.  We conclude that $\alpha$ is inner and 
$\gamma \in Z_G \times r(K_2)$.
\end{proof}

Finally we come to the main result of this section.

\begin{theorem}\label{cw}
Suppose that $\rank K = \rank G$.
If $\Gamma$ is a group of isometries of constant displacement on
$\widetilde{M}$ then $\Gamma \subset (Z_G \times r(K_2))$ where 
$Z_G$ denotes the center of $G$.  Conversely, if $\Gamma \subset 
(Z_G \times r(K_2))$ then every $\gamma \in \Gamma$ is an isometry
of constant displacement on $\widetilde{M}$.
\end{theorem}
\begin{proof} $\Gamma \subset I(\widetilde{M})$, so Theorem \ref{isogroup}
says $\Gamma \subset \widetilde{G}^\dagger =
{\bigcup}_{\alpha \in \Out(G,K_1,K_2)}\, \widetilde{G} \alpha$.  Lemma
\ref{outers} implies $\Gamma \subset \widetilde{G}$, and from Proposition
\ref{cw-conn} we conclude that $\Gamma \subset (Z_G \times r(K_2))$\,.
Conversely, if $\Gamma \subset (Z_G \times r(K_2))$ then $G$ centralizes
$\Gamma$ so every $\gamma \in \Gamma$ is of constant displacement.
\end{proof}

\begin{corollary}\label{cw-cover}
Let $\widetilde{M} \to \Gamma\backslash \widetilde{M}$ be a riemannian
covering whose deck transformation group $\Gamma$ consists of isometries
of constant displacement.  Then $\Gamma \subset (Z_G \times r(K_2))$ and
$\Gamma\backslash \widetilde{M}$ is homogeneous.
\end{corollary}

\section{Isometries of Constant Displacement: Case $\chi(M) = 0$}
\label{sec6}
\setcounter{equation}{0}
In this section we study the cases where $\chi(M) = 0$, in other words
where $\rank K < \rank G$.  We know that the identity component
$I^0(\widetilde{M}) = G \times r(K_2)$ by Theorem \ref{ident-comp}.
We will prove the following analog of Proposition \ref{cw-conn} for 
$\rank K < \rank G$.  This uses an argument of C\' ampoli \cite{C1986}.

\begin{theorem}\label{cw-less}
Let $\pi: \widetilde{M} \to M$ as in {\rm (\ref{setup})} with
$\chi(M) = 0$.  If $\Gamma$ is a group of isometries of constant 
displacement on $\widetilde{M}$, and if $\Gamma \subset I^0(\widetilde{M})$,
then $\Gamma \subset (Z_G \times r(K_2))$.  Conversely if 
$\Gamma \subset (Z_G \times r(K_2))$ then every
$\gamma \in \Gamma$ is an isometry of constant displacement on $\widetilde{M}$.
\end{theorem}

\begin{proof} Let $\gamma = (g,r(k_2) \in \Gamma$.  It descends to an
isometry $g$ of $M$.  If $g$ has a fixed point on $M$, in other words if 
it preserves a fiber of $\pi: \widetilde{M} \to M$, then the argument of 
Proposition \ref{cw-conn} proves $\gamma \in Z_G \times r(K_2)$.
Now suppose that $g$ does not have a fixed point on $M$.
\medskip

Let $t \to \sigma(t)$ denote the minimizing geodesic in $\widetilde{M}$
from $1K_1$ to $\gamma(1K_1)$.  Then $\sigma(t) = \exp(t\xi)(1K_1)$
where $\xi = \xi_1 + dr(\xi_2)$ with $\xi_1 \in \gg$, $\xi_1 \perp \gk_1$\,,
and $\xi_2 \in \gk_2$.  Here $\xi_1$ belongs to the Lie algebra $\gt_1$
of a maximal torus $T_1$ of $G$ such that $\gt_1 = \gt'_1 + \gt''_1$ where
$\xi \in \gt'_1$\,, $\gt'_1 \perp \gk_1$\,, and $\gt''_1$ is the Lie
algebra of a maximal torus of $K_1$\,.  The isometry $\gamma$ has constant
displacement equal to $||\xi_1 + dr(\xi_2)||$.  Note that $\xi_1 \ne 0$
because $g$ does not have a fixed point on $M$.
\medskip

Every conjugate of $\gamma$ has the same constant displacement.  In
particular if $w$ belongs to the Weyl group $W(G,T_1)$ then
$||\xi_1 + dr(\xi_2)|| = ||p(w(\xi_1)) + dr(\xi_2)||$ where
$p : \gt_1 \to \gt'_1$ is orthogonal projection.  As $\xi_1 \perp
dr(\xi_2) \perp p(w(\xi_1))||$ it follows that $||\xi_1|| = ||p(w(\xi_1))||$.
From that, $w(\xi_1) \in \gt'_1$ for every $w \in W(G,T_1)$.  But
$W(G,T_1)$ acts irreducibly on $\gt_1$ so it cannot preserve the subspace
$\gt'_1$\,.  This implies $\xi_1 = 0$.  That contradicts the
assumption that $g$ has no fixed point on $M$.  In other words,
$\gamma \in (Z_G \times r(K_2))$, as asserted.
The theorem follows.
\end{proof}

In the next section we look at the special case of Stieffel manifold
fibrations over odd dimensional Grassmann manifolds.

\section{Isotropy--Split Bundles over Odd Real Grassmannians}
\label{sec7}
\setcounter{equation}{0}
In this section we extend the theory described in Sections \ref{sec3},
\ref{sec5} and \ref{sec6} to include isotropy--split bundles 
$\pi: \widetilde{M} \to M$
where the base $M$ is an irreducible compact riemannian symmetric space 
$G/K$ such that $\rank G > \rank K$.  According to the classification, the 
only such $G/K$ are
{\footnotesize
$$
SU(n)/SO(n), SU(2n)/Sp(n), 
SO(2s+2+2t)/[SO(2s + 1)\times SO(1+2t)], E_6/F_4, E_6/Sp(4), 
(K\times K)/diag(K).
$$
}
Note that $SU(4)/SO(4) = SO(6)/[SO(3) \times SO(3)]$.  Thus the only such
symmetric spaces $G/K$ that satisfy (\ref{setup}) are the oriented
real Grassmann manifolds $SO(2s+2t+2)/[SO(2s + 1)\times SO(2t+1)]$ of odd
real dimension.  Thus
we look at
\begin{equation}\label{odd-fibr}
\begin{aligned}
\pi: \widetilde{M} \to M &\text{ given by } G/K_1 \to G/K_1K_2 \text{ where }\\
     &G = SO(2s+2+2t), K_1 = SO(2s + 1) \text{ and } K_2 =  SO(1+2t).
\end{aligned}
\end{equation}
The odd spheres are completely understood (\cite{W1961} and \cite{W1966}),
and in any case they do not lead to isotropy--split fibrations, so
we put those cases aside and assume $s, t > 0$.
\medskip

Theorem \ref{ident-comp} and the first three statements of Corollary
\ref{CK-fields} are valid here.
We still have the relative groups (\ref{outer}) and the isometry groups
(\ref{daggers}), but neither $K_1$ nor $K_2$ has an 
outer automorphism.  However, following Cartan, the symmetry $s$ at the 
base point $1K$ of $G/K$ gives another component of $I(\widetilde{M})$.  In
fact it is clear that 
$s = \left ( \begin{smallmatrix} I_{2s+1} & 0 \\ 0 & -I_{1+2t} 
\end{smallmatrix} \right )$.  Thus

\begin{proposition}\label{iso-st}
The full isometry group $I(\widetilde{M})$ is the $2$--component group
$O(2s+2+2t) \times r(SO(1+2t))$.  The set of all constant length Killing 
vector fields on $\widetilde{M}$ is $dr(\gs\go(1+2t))$. 
\end{proposition}

\begin{lemma}\label{fp-symm}
Every element of $sI^0(\widetilde{M})$ has a fixed point on $M$.
\end{lemma}
\begin{proof} If $g \in I^0(\widetilde{M})$ then the matrix $sg$ has
determinant $-1$.  Let $R(\theta)$ denote the rotation matrix
$\left ( \begin{smallmatrix} \cos(\theta) & \sin(\theta) \\
-\sin(\theta) & \cos(\theta) \end{smallmatrix} \right )$. Then 
$\R^{2s+2+2t}$ has an orthonormal basis $\{e_i\}$
in which $sg$ has matrix 
$$
\diag\{R(\theta_1), \dots ,  R(\theta_s),1,
-1,R(\theta'_1), \dots , R(\theta'_t)\}.
$$  
Now $\Span\{e_1, \dots , e_{2s+1}\}$ is the fixed point.
\end{proof}

Combining Theorem \ref{cw-less} with Proposition \ref{iso-st} and 
Lemma \ref{fp-symm} we have

\begin{theorem}\label{cw-st} 
Let $\pi: \widetilde{M} \to M$ as in {\rm (\ref{odd-fibr})}.  If $\Gamma$
is a group of isometries of constant displacement on the Stieffel manifold
$\widetilde{M}$ then $\Gamma \subset (\{\pm I\} \times r(SO(1+2t)))$.
Conversely if $\Gamma \subset (\{\pm I\} \times r(SO(1+2t)))$ then every
$\gamma \in \Gamma$ is an isometry of constant displacement on $\widetilde{M}$.
\end{theorem}

Now, as in Corollary \ref{cw-cover}, we have

\begin{corollary}\label{cw-cover-st}
Let $\pi: \widetilde{M} \to M$ be the isotropy-split fibration 
{\rm (\ref{odd-fibr})}.  
Let $\widetilde{M} \to \Gamma\backslash \widetilde{M}$ be a riemannian
covering whose deck transformations $\gamma \in \Gamma$ 
have constant displacement.  Then 
$\Gamma \subset (\{\pm I\} \times r(SO(1+2t)))$ and 
$\Gamma\backslash \widetilde{M}$ is a riemannian homogeneous space.
\end{corollary}

\section{Applications to the Homogeneity Conjecture}
\label{sec8}
\setcounter{equation}{0}
The background to the Homogeneity Conjecture consists of three
papers from the 1960's 
concerning riemannian coverings $S \to \Gamma \backslash S$ where
$S$ is a riemannian homogeneous space.  The first one, \cite{W1960},
studies the case where $S$ has constant sectional curvature and
classifies the quotients $\Gamma \backslash S$ that are riemannian
homogeneous.  There it is shown that if $\Gamma \backslash S$ is
homogeneous then every $\gamma \in \Gamma$ is of constant displacement.
The second one, \cite{W1961}, also for the case of constant sectional
curvature, shows that if every $\gamma \in \Gamma$ is of constant displacement
then $\Gamma \backslash S$ is homogeneous.  The third one, \cite{W1962},
extends these results to symmetric spaces: Let $S$ be a connected simply 
connected riemannian symmetric space and  $S \to \Gamma \backslash S$ a
riemannian covering; then $\Gamma \backslash S$ is homogeneous if and
only if every $\gamma \in \Gamma$ is of constant displacement.  Thus 
\medskip

\noindent {\bf Homogeneity Conjecture.}
{\sl Let $S$ be a connected simply
connected riemannian homogeneous space and  $S \to \Gamma \backslash S$ a
riemannian covering.  Then $\Gamma \backslash S$ is homogeneous if and
only if every $\gamma \in \Gamma$ is of constant displacement.}
\medskip

We note that one direction of the Homogeneity Conjecture is easy: If
$\Gamma \backslash S$ is homogeneous then $I^0(\Gamma \backslash S)$ lifts
to a subgroup $G$ of $I(S)$, and $G$ normalizes $\Gamma$ by construction.
Then $G$ centralizes $\Gamma$ because $G$ is connected and $\Gamma$ is
discrete.  Also, $G$ is transitive on $S$ because $\Gamma \backslash S$ is 
homogeneous.  If $x, y \in S$ and $\gamma \in \Gamma$ choose $g \in G$ with
$y = gx$.  Then the displacement $\delta_\gamma(x) = {\rm dist}(x,\gamma x)
= {\rm dist}(gx,g\gamma x) = {\rm dist}(gx, \gamma gx) = {\rm dist}(y,\gamma y)
= \delta_\gamma(y)$.  Thus if $\Gamma \backslash S$ is homogeneous then
every $\gamma \in \Gamma$ is of constant displacement.  The hard part is
the converse.
\medskip

Since the Homogeneity Conjecture was proved for $S$ riemannian symmetric,
some other cases of the conjecture have been proved.  The latest cases
are those of Corollaries \ref{cw-cover} and
\ref{cw-cover-st}:

\begin{theorem}\label{conj-isotropy--split}
Let $\widetilde{M} \to M$ be an isotropy--splitting fibration 
and let $\widetilde{M} \to \Gamma\backslash \widetilde{M}$ be a riemannian
covering.  Suppose that $\rank K = \rank G$, or that $M$ is a riemannian 
symmetric space, or that $\Gamma \subset I^0(\widetilde{M})$.  
Then $\Gamma\backslash \widetilde{M}$ is homogeneous if and only
if every $\gamma \in \Gamma$ is an isometry of constant displacement.
\end{theorem}
\noindent

Now we try to describe the broader mathematical context of 
Theorem \ref{conj-isotropy--split}.  There are five lines of research there:
(i) decreasing the number of case by case verifications of \cite{W1962}, 
(ii) dealing with nonpositive curvature and bounded isometries, 
(iii) additional special cases where $S$ is compact, 
(iv) Killing vector fields of constant length, and
(v) extension of these results from riemannian to Finsler manifolds..
\medskip

{\bf Concerning Case by Case Verifications.} 
My proof \cite{W1962} of the Homogeneity Conjecture for symmetric 
spaces involved a certain amount of case by case verification.  Some of 
this was simplified later by 
Freudenthal \cite{F1963} and Ozols (\cite{O1969}, \cite{O1973}, \cite{O1974}) 
with the restriction that $\Gamma$ be contained in the identity component 
of the isometry group.
\medskip

{\bf Nonpositive Curvature and Bounded Isometries.}
This approach was implicit in the treatment of symmetric spaces of
noncompact type in \cite{W1960} and \cite{W1962}, and extended in \cite{W1964}
to all riemannian manifolds $S$ of non--positive sectional curvature.  The 
idea is to apply an isometry $\gamma$ of bounded displacement to a geodesic
$\sigma$ and see that $\sigma$ and $\gamma(\sigma)$ bound a flat totally
geodesic strip in $S$.  This was extended later by Druetta \cite{D1983} to
manifolds without focal points.  Further evidence for the Homogeneity
Conjecture was developed by Dotti, Miatello and the author in \cite{DMW1986}
for riemannian manifolds that admit a transitive semisimple group of isometries 
that has no compact factor, and by the author in \cite{W2016} for
riemannian manifolds that admit a transitive exponential solvable group
of isometries. Here also see \cite{W1963}.
\medskip

{\bf Killing Vector Fields of Constant Length.}
The infinitesimal version of isometries of constant displacement is that of
Killing vector fields of constant length.  This topic seems to have been
initiated by Berestovskii and Nikonorov in (\cite{BN2008a}, \cite{BN2008b}, 
\cite{BN2009}), and was further developed by Nikonorov 
(\cite{N2013}, \cite{N2016}), and by Podest\` a, myself and Xu (\cite{WPX2016},
\cite{WX2016a}, \cite{WX2016b}).  Also see Corollary \ref{CK-fields} above.
\medskip

{\bf Finsler Extensions.}
If $(M,F)$ is a Finsler symmetric space, say $M = G/K$ where $G$ is the
identity component of the (Finsler) isometry group, then there is a
$G$--invariant riemannian metric $ds^2$ on $M$ such that $(M,ds^2)$ is
riemannian symmetric with the same geodesics as $(M,F)$.  See
\cite[Theorem 11.6.1]{W2007} and the discussion in \cite[\S 11.6]{W2007}.
Deng and I proved the Homogeneity Conjecture for Finsler symmetric spaces
in \cite{DW2012} by reduction to the riemannian case.  Further, Deng and
others in his school, especially Xu, have done a lot on isometries of 
constant displacement and on Killing vector fields of constant length;
for example see \cite{DX2012}, \cite{DX2013a}, \cite{DX2013b}, 
\cite{DX2014a}, \cite{DX2014b} and \cite{DX2015}.  The arguments involving
reduction to the riemannian case usually depend on very technical
computations.  In \cite{DW2012}, for example, one has to prove the 
Berwald condition in order to get around the lack
of a de Rham decomposition for Finsler manifolds.

\section{Isotropy--Splitting Fibrations over Noncompact Spaces}
\label{sec9}
\setcounter{equation}{0}
In this section we examine our theory of isotropy--splitting fibrations 
$\pi: \widetilde{M} \to M$ in the setting in which the base $M$ is noncompact.
For example $M$ could be the noncompact
dual of one of the riemannian symmetric spaces of Section \ref{sec2}, or
a certain variation for the $3$--symmetric and $5$--symmetric spaces.
\medskip

Here is the basic problem with noncompact base manifolds.
Recall that $M = G/K$ and $\widetilde{M} = G/K_1$ carry the normal metrics 
defined by the negative $-\kappa$ of the Killing form of $\gg$.  
Then $-\kappa$ cannot be definite on $\gk_1^\perp$\,,
for $\gk_1$ cannot be a maximal compactly embedded subalgebra.
So we are forced to either restrict attention to the setting of compact
riemannian manifolds $M$ and $\widetilde{M}$, or expand attention to
the situation where $\widetilde{M}$ is a noncompact pseudo--riemannian 
manifold.  At that point we modify the compact manifold definition
(\ref{setup}) for isotropy--splitting fibrations $\pi: \widetilde{M} \to M$,
replacing ``compact'' by ``reductive'' and dealing with the lack of a general
de Rham decomposition for pseudo--riemannian manifolds:
\begin{equation}\label{setup-ps}
\begin{aligned}
&G \text{ is a connected real reductive linear algebraic group with } 
	G_\C \text{ simply connected,}\\
&K = K_1K_2 \text{ where the $K_i$ are closed connected reductive algebraic
	subgroups of $G$ such that}\\
& \phantom{XXXX}\text{ (i)  $K = (K_1 \times K_2)/(K_1 \cap K_2)$,}
        \text{ (ii) } \gk_2 \perp \gk_1 \text{ and }
        \text{ (iii) } \dim \gk_1 \ne 0 \ne \dim \gk_2\,, \\
&\text{the centralizers } Z_G(K_1) = Z_{K_1}\widetilde{K_2}
        \text{ and } Z_G(K_2) = Z_{K_2}\widetilde{K_1} \text{ with }
        K_1 = \widetilde{K_1}^0 \text{ and }
        K_2 = \widetilde{K_2}^0\,, \text{ and } \\
&M = G/K \text{ and } \widetilde{M} = G/K_1 \text{ are
	normal pseudo--riemannian homogeneous spaces of } G.
\end{aligned}
\end{equation}
As before we may assume that the metrics on $M$ and $\widetilde{M}$ are the
normal pseudo--riemannian metrics defined by the negative of the Killing 
form of $G$.

\subsection{Noncompact Riemannian Symmetric Base}
\label{ssec9a}

A particularly interesting case is where $M = G/K$ is an irreducible riemannian
symmetric space of noncompact type.  We list all such isotropy--splitting 
fibrations $\pi: \widetilde{M} \to M$, given by $G/K_1 \to G/K$ with
$\rank K = \rank G$.  There of course $M = G/K$ is the noncompact dual of one 
of the fibrations of Section \ref{sec2} over a compact symmetric space of
nonzero Euler characteristic.
The ones with hermitian symmetric space base are characterized by
$K = SK'$ where $S$ is a circle group and $K' = [K,K]$
is semisimple.  The corresponding fibrations are
$$
\begin{aligned}
&G/K' \to G/K \text{ circle bundle over a bounded symmetric domain and } \\
&G/S \to G/K \text{ principal $K'$--bundle over bounded symmetric domain.}
\end{aligned}
$$
In addition, if $\gg = \gs\gu(s,t)$ then $\gk' = \gs\gu(s)\oplus\gs\gu(t)$,
leading to fibrations
$$
\begin{aligned}
&SU(s,t)/SU(s) \to SU(s,t)/S(U(s)U(t)) \text{ and  }
        SU(s,t)/U(s) \to SU(s,t)/S(U(s)U(t)),\\
&SU(s,t)/SU(t) \to SU(s,t)/S(U(s)U(t)) \text{ and  }
        SU(s,t)/U(t) \to SU(s,t)/S(U(s)U(t)).
\end{aligned}
$$
When the base $M = G/K$ is a nonhermitian symmetric space, $K$ is
simple except in the cases
$$
\begin{aligned}
& G/K = SO^0(s,t)/SO(s)SO(t) \text{ with } 2 < s \leqq t
        \text{ and } st \text{ even, }\\
& G/K = Sp(s,t)/Sp(s)Sp(t) \text{ with } 1 \leqq s \leqq t, \\
& G/K = G_{2,A_1A_1}/A_1A_1\,, G/K = F_{4,A_1C_3}/A_1C_3\,, 
E_{6,A_1A_5}/A_1A_5\,, E_{7,A_1D_6}/A_1D_6 \text{ or }
        E_{8,A_1E_7}/A_1E_7\,.
\end{aligned}
$$
In the $SO$ cases, $G/K$ is a quaternion--Kaehler symmetric space for 
$s = 3$ and for $s = 4$.  In the $Sp$ cases $G/K$ is a quaternion--Kaehler 
symmetric space for $s =1$.  In the exceptional group cases $G/K$ always 
is a quaternion--Kaehler symmetric space.
\medskip

Finally we list the isotropy--splitting fibrations $\pi: \widetilde{M} \to M$, 
given by $G/K_1 \to G/K$ with $\rank K < \rank G$.  There $M = G/K$ is the
noncompact dual of an odd dimensional real Grassmann manifold
$SO(2s+2+2t)/[SO(2s+1) SO(1+2t)]$ with $s, t > 1$, leading to the
fibrations
$$
\begin{aligned}
& G/K = SO^0(2s+1,1+2t)/SO(2s+1) \to SO^0(2s+1,1+2t)/[SO(2s+1) SO(1+2t)]
		\text{ and }\\
& G/K = SO^0(2s+1,1+2t)/SO(1+2t) \to SO^0(2s+1,1+2t)/[SO(2s+1) SO(1+2t)]
\end{aligned}
$$
with $s,t > 1$.
\subsection{Compact Riemannian Dual Fibration}
\label{ssec9b}

Given $(G,K_1,K_2)$ as in (\ref{setup-ps}), there is a Cartan involution 
$\theta$ 
of $G$ that preserves each $K_i$ and restricts on it to a Cartan involution.
That defines the compact Cartan dual triple $(G^u,K_1^u,K_2^u)$ and
the compact riemannian isotropy--splitting fibration 
$\pi^u: \widetilde{M^u} \to M^u$\,,
given by $G^u/K^u_1 \to G^u/K^u_1K^u_2$\,, as in (\ref{setup}).
Several pseudo--riemannian isotropy--splitting fibrations 
$\pi: \widetilde{M} \to M$ can
define the same $\pi^u: \widetilde{M^u} \to M^u$\,. We say that
$\pi: \widetilde{M} \to M$ is {\em associated} to 
$\pi^u: \widetilde{M^u} \to M^u$\,.
\medskip

One special case is the one where $G$, $K_1$ and $K_2$ each is the underlying
real structure of a complex Lie group.  Then $G = G^u_\C$,, $K_1 = (K_1)^u_\C$
and $K_2 = (K_2)^u_\C$\,.
\medskip

For each of the isotropy--splitting fibrations $\pi^u: \widetilde{M^u} \to M^u$ 
that satisfies (\ref{setup}) we find all associated pseudo--riemannian
isotropy--splitting fibrations $\pi: \widetilde{M} \to M$ in the tables of
\cite{WG1968a} and \cite{WG1968b}. Here, for example, are the ones for
$n_0 = 3$, corresponding to the nearly--Kaehler base spaces of Section 
\ref{ssec2c}, taken from \cite[Table 7.13]{WG1968b}.
\medskip

\noindent
$\bullet$ For $\widetilde{M^u} \to M^u$ where $M^u = F_4/A_2A_2$, $M$ can be
$M^u$ or one of
$$
\begin{aligned}
&F_{4,B_4}/[SU(1,2)SU(3)]\,\,, F_{4,C_1C_3}/[SU(3)SU(1,2)],\\
&F_{4,C_1C_3}/[SU(1,2)SU(1,2)]\, \text{ or }\, F_4^\C/[SL(3;\C)SL(3;\C)].
\end{aligned}
$$
\noindent
$\bullet$ For $\widetilde{M^u} \to M^u$ where $M^u = E_6/A_2A_2A_2$, $M$ can be
$M^u$ or one of
$$
\begin{aligned}
&E_6/A_1A_5/[SU(1,2)SU(3)SU(3)],\,\, E_6/A_1A_5/[SU(1,2)SU(1,2)SU(3)],\\
&E_{6,D_5T^1}/[SU(1,2)SU(1,2)SU(3)], \,\text{ or }\, 
E_6^\C/[SL(3;\C)SL(3;\C)SL(3;\C)].
\end{aligned}
$$
$\bullet$ For $\widetilde{M^u} \to M^u$ where $M^u = E_7/A_2A_5$, $M$ can be
$M^u$ or one of
$$
\begin{aligned}
&E_{7,A_1D_6}/[SU(1,2)SU(1,5)],\,\, E_{7,A_1D_6}/[SU(3)SU(2,4)],\,\,
	E_{7,A_1D_6}/[SU(1,2)SU(2,4)], \\
&E_{7,E_6T^1}/[SU(1,2)SU(1,5)],\,\, E_{7,E_6T^1}/[SU(3)SU(3,3)], \,\text{ or }\,
	E_7^\C/[SL(3;\C)SL(6;\C)].
\end{aligned}
$$
$\bullet$ For $\widetilde{M^u} \to M^u$ where $M^u = E_8/A_2E_6$, $M$ can be
$M^u$ or one of
$$
\begin{aligned}
&E_{8,A_1E_7}/[SU(1,2)E_6],\,\, E_{8,A_1E_7}/[SU(1,2)E_{6,D_5T^1}],\\
&E_{8,A_1E_7}/[SU(3)E_{6,A_1A_5}], \,\text{ or }\, E_8^\C/[SL(3;\C)E_6^\C].
\end{aligned}
$$

\subsection{Isometries and Killing Vector Fields}
\label{ssec9c}

We use the notation (\ref{setup-ps}).  As in the compact case we have
\begin{equation}\label{g-tilde-nonc}
\widetilde{G} := G \times r(K_2) \text{ is connected and algebraic, and 
acts on } \widetilde{M} \text{ by } (g,r(k_2)): xK_1 \mapsto gxk_2^{-1}K_1\,.
\end{equation}
Theorem \ref{ident-comp} extends to the pseudo--riemannian setting as follows.
\begin{proposition}\label{ident-comp-ps}
If $\widetilde{M}$ is irreducible then
the isometry group of $\widetilde{M}$ has identity component
$I^0(\widetilde{M}) = \widetilde{G}$.
\end{proposition}

\begin{proof} Evidently $\widetilde{G}$ acts by isometries on
$\widetilde{M}$, and by hypothesis $\widetilde{M}$ is an affine algebraic
variety.  
Now suppose that $\widetilde{G} \subsetneqq L$ where $L$ is a closed 
connected algebraic subgroup of $I(\widetilde{M})$.  Let $L^{red}$ denote 
a maximal reductive subgroup of $L$ that contains $\widetilde{G}$, so 
$\widetilde{G} \subset L^{red}$.  The compact real forms 
$\widetilde{G}^u \subset L^{red,u}$, and Theorem \ref{ident-comp} ensures 
that $\widetilde{G}^u = I^0(\widetilde{M}^u) = L^{red,u}$.  Thus 
$\widetilde{G}$ is a maximal reductive subgroup of $L$.
\medskip

Let $H$ denote the isotropy subgroup of $L$ at $1K_1$\,.  It contains
$\widetilde{K_1} := K_1 \times \{(k_2,r(k_2)) \mid k_2 \in K_2\}$, the 
isotropy subgroup of $\widetilde{G}$ at $1K_1$\,, and $\widetilde{K_1}$ 
is its maximal reductive subgroup.  Let $L^{unip}$ denote the unipotent
radical of $L$.  Since $L^{unip}$ is a normal subgroup of $L$ its orbits
satisfy $L^{unip}(gK_1) = gL^{unip}(K_1)$.  Thus 
$\widetilde{M} \to L^{unip} \backslash \widetilde{M}$ would be a fiber
space if the $L^{unip}$--orbits on $\widetilde{M}$ were closed 
submanifolds.  To get around that problem let  $N$ denote the categorical 
quotient $L^{unip} \backslash\backslash \widetilde{M}$.  We can view it
as the base space of the fibration whose fibers are the closures of
the $L^{unip}$--orbits on $\widetilde{M}$.  The (transitive) action of
$\widetilde{G}$ on $\widetilde{M}$ descends to a smooth transitive action
of $\widetilde{G}$ on $N$.  The action of $L$ descends as well. 
Write $N = \widetilde{G}/Q$ where $Q$ contains $\widetilde{K_1}$.  
Then $\widetilde{K_1} L^{unip}$ is the isotropy subgroup of $L$ on $N$
and $Q = \widetilde{G}\cap (\widetilde{K_1} L^{unip}) = \widetilde{K_1}$\,.
This says $\widetilde{M} \to L^{unip} \backslash\backslash \widetilde{M}$
is one to one.  In other words the action of $L^{unip}$ on $\widetilde{M}$ is
trivial.  As $L$ acts effectively by its definition, $L^{unip} = \{1\}$.
Now $L = L^{red} = \widetilde{G}$.
\end{proof}

In the pseudo--riemannian setting we don't have a good notion for the
displacement of an isometry, but we still have its infinitesimal analog.
We define {\em constant length} for a vector field to mean constant inner 
product with itself relative to the invariant pseudo--riemannian metric.  
Fix a Cartan involution $\theta$ of $\widetilde{G}$ that preserves
$G$, $K_1$\,, $K$ and $\widetilde{K_1}$\,.  We say that an element
$\xi \in \widetilde{\gg}$ is {\em elliptic} if $d\theta(\Ad(g)\xi) = \Ad(g)\xi$
for some $g \in G$, {\em hyperbolic} if $d\theta(\Ad(g)\xi) = -\Ad(g)\xi$.
for some $g \in G$.  In other words $\xi$ is elliptic if all the eigenvalues
of $\ad(\xi)$ are pure imaginary, hyperbolic if all the eigenvalues of $\xi$
are real.  Here is the noncompact base analog of of Corollary \ref{CK-fields}:

\begin{corollary}\label{CK-fields-nonc}
Suppose that 
the pseudo--riemannian manifold $\widetilde{M} = G/K_1$ is irreducible.
\begin{itemize}
\item[{\rm (1)}] The algebra of all Killing vector fields on $\widetilde{M}$ is
$\widetilde{\gg} = \gg \oplus dr(\gk_2)$.
\item[{\rm (2)}] The set of all constant length
Killing vector fields on $\widetilde{M}$ is $\gl \oplus dr(\gk_2)$ where \\
\phantom{XXXXXXXXXX}
$\gl = \{\xi \in \gg \mid \xi \text{ defines a constant length Killing
vector field on } \widetilde{M}\}.$
\item[{\rm (3)}]
$\gl = \{\xi \in \gg \mid \xi \text{ defines a constant length Killing
vector field on } M \}.$
\item[{\rm (4)}]
If $\rank K = \rank G$ and $\xi \in \gl$ then $\xi$ is $\Ad(G)$--conjugate
to an element of $\gk$ and the corresponding Killing vector field 
has norm $||\xi_{gK_1}|| = 0$ at every point $gK_1 \in \widetilde{M}$.
In particular $\gl$ does not contain a nonzero elliptic element nor
a nonzero hyperbolic element.
\item[{\rm (5)}]
If $\rank K = \rank G$ and $K$ is compact, 
then $\gl = 0$, so $dr(\gk_2)$ is the set of all 
constant length Killing vector fields on $\widetilde{M}$.
\end{itemize}
\end{corollary}

\begin{proof}  The first statement is immediate from Theorem 
\ref{ident-comp-ps}, so we turn to the second and third.  
\medskip

If $\xi$ is a Killing vector field on $\widetilde{M}$
then $\xi \in \gg \oplus dr(\gk_2)$.  Decompose $\xi = \xi' + \xi''$
with $\xi' \in \gg$ and $\xi'' \in dr(\gk_2)$.  Then $\xi''$ has constant
length because the corresponding vector field on $\widetilde{M}$ is 
invariant under the transitive isometry group $G$.  The vector fields of
$\xi'$ and $\xi''$ are orthogonal at $1K_1$\,.  It follows that they are
orthogonal at every point of $\widetilde{M}$ because $\xi'$ is orthogonal
to the fibers of $\widetilde{M} \to M$ at every point of $\widetilde{M}$.  
If $\xi$ has constant length now $\xi'$ also has constant length.   
That proves the second statement.
In the argument just above, $\xi$ and $\xi'$ define the same Killing
vector field on $M$.  The third statement follows.
\medskip

Now suppose $\rank K = \rank G$ and let $\xi \in \gl$.  Using a Cartan
involution of $\gg$ we write $\xi = \xi_{ell} + \xi_{hyp}$ where
$\xi_{ell}$ is elliptic and $\xi_{hyp}$ is hyperbolic.  Recall that we are 
using the negative of the Killing form of $\gg$ for the pseudo--riemannian 
metrics both on $\widetilde{M}$ and $M$.  If the square length $||\xi||^2 > 0$
on $\widetilde{M}$ then $\xi_{ell}$ never vanishes on $M$, contradicting 
$\rank K = \rank G$.
If $||\xi||^2 < 0$ then $\xi_{hyp}$ never vanishes on $M$, contradicting
$\rank K = \rank G$.  Thus $||\xi||^2 = 0$ on $\widetilde{M}$, and thus
on $M$.  As above $\xi$ has a zero on $M$, in other words some 
$\Ad(G)$--conjugate 
of $\xi$ belongs to $\gk$.  Thus $\xi \in \gk$ with length $||\xi|| = 0$
at every point of $\widetilde{M}$  In particular, if $K$ is compact then
$\xi$ is elliptic, so $\xi = 0$ and $dr(\gk_2)$ is the set of all
constant length Killing vector fields on $\widetilde{M}$.
\end{proof}

\begin{corollary}\label{fiber-inv-ps}
Suppose that $\widetilde{M}$ is irreducible.  Then every isometry of
$\widetilde{M}$ sends fiber to fiber in the isotropy--split fibration
$\widetilde{M} \to M$, and thus induces an isometry of $M$.
\end{corollary}
\begin{proof}
In view of Corollary \ref{CK-fields-nonc}(2) and \ref{CK-fields-nonc}(4),
and because we have a Cartan involution $\theta$ of $\widetilde{G}$
such that $\theta(K_2) = K_2$,
the tangent space to the fiber $r(K_2)(gK_1)$ is the span of all vector
fields $\xi_e + \xi_h$ where $\xi_e \in dr(\gk_2)$ is elliptic and
$\xi_h \in dr(\gk_2)$ is hyperbolic.
\end{proof}

Now we can look for the full isometry group of $\widetilde{M}$.  In fact
the result is very close to results in \cite[Section 2]{DMW1986}, and
the way we use it is contained in \cite[Section 2]{DMW1986}, but the
argument here is closer to the structure of the isotropy--splitting fibration.
Let
\begin{equation}\label{out-nonc}
\Out(G,K_1) = \{\alpha \in \Out(G) \mid \alpha(K_1) = K_1 \text{ and }
        \alpha|_{K_1} \in \Out(K_1)\} \subset \Out(G,K).
\end{equation}
We define two subgroups of isometry groups as in (\ref{daggers}) by
\begin{equation}\label{daggers-nonc}
G^\dagger = {\bigcup}_{\alpha \in \Out(G,K_1)}\, G \alpha \subset I(M)
\text{ and } \widetilde{G}^\dagger =
{\bigcup}_{\alpha \in \Out(G,K_1)\,, \beta \in \Out(G,K_2)}\,
G\alpha \cdot r(K_2)\beta \subset I(\widetilde{M}).
\end{equation}
As before, $g\alpha$ acts on $M$ by $xK \mapsto g\alpha(x)K$ and on
$\widetilde{M}$ by  $xK_1 \mapsto g\alpha(x)K_1$\,, and $r(k_2)\beta$
acts on $\widetilde{M}$ by $xK_1 \mapsto x\beta(k_2)^{-1}K_1$\,.
\medskip

\begin{theorem}\label{isogroup-nonc}
Let $\pi: \widetilde{M} \to M$ be an isotropy--split fibration as in
{\rm (\ref{setup-ps})}.  Suppose $\rank K = \rank G$.  Then
the identity component $I^0(\widetilde{M}) = \widetilde{G}$ and
the full isometry group $I(\widetilde{M}) = \widetilde{G}^\dagger$.
\end{theorem}

\begin{proof} 
As $\widetilde{G}$ is a reductive linear 
algebraic group every component of $\Aut(\widetilde{G})$ contains an 
elliptic element.  Thus every component of $\Aut(\widetilde{G})$ has
an element in common with $\Aut(\widetilde{G}^u)$.  Corollary \ref{fiber-inv-ps}
carries this down to $M$.  That gives injections
$I(\widetilde{M})/I^0(\widetilde{M}) \hookrightarrow
I(\widetilde{M^u})/I^0(\widetilde{M^u})$ and $I(M)/I^0(M) \hookrightarrow
I(M^u)/I^0(M^u)$, so Lemmas \ref{inn-out} and \ref{isotropy} extend 
to our pseudo--riemannian setting.  Our assertions follow by combining
Theorem \ref{isogroup} with Proposition \ref{ident-comp-ps}.
\end{proof}

While don't have a notion of constant displacement here, we can at least
study homogeneity for pseudo--riemannian coverings by $\widetilde{M}$.

\begin{corollary}\label{homog-q-ps}
Let $\pi: \widetilde{M} \to M$ be an isotropy--split fibration as in
{\rm (\ref{setup-ps})}.  Suppose $\rank K = \rank G$.  Let $p$
denote the projection $\widetilde{G}^\dagger \to G^\dagger$ of
$I(\widetilde{M})$ into $I(M)$.  

{\rm (1)} Let $\gamma \in I(\widetilde{M})$ such that $p(\gamma)$ is 
elliptic and the centralizer of $\gamma$ is transitive on $\widetilde{M}$.
Then $p(\gamma) \in Z_G$\,. 

{\rm (2)} Consider a pseudo--riemannian covering
$\widetilde{M} \to \Gamma\backslash \widetilde{M}$ such that $p(\Gamma)$
has compact closure in $I(M)$  $($for example such that $p(\Gamma)$ is
finite$)$.  Then $\Gamma\backslash \widetilde{M}$ is homogeneous if and
only if $p(\Gamma) \subset Z_G$\,.
\end{corollary}

\begin{proof}  Let $J$ denote the centralizer of $p(\gamma)$ in $G^\dagger$. 
Then $\gg = \gj + \gk$.  As $p(\gamma)$ is semisimple $\gg^u = \gj^u + \gk^u$ so
$J^u$ is transitive on $M^u$.  Also, $p(\gamma) \in G^u$ because it is elliptic.
Now $p(\gamma)$ has constant displacement on $M^u$ and the assertion follows 
from Theorem \ref{cw}.
\end{proof}

\subsection{Isotropy--Split Fibrations over Odd Indefinite Symmetric Spaces}
\label{ssec9d}

We now deal with the cases where $\rank K < \rank G$ and $M = G/K$ is a
pseudo--riemannian symmetric space.  Here we follow the lines of Section 
\ref{sec7}.
\medskip

According to the classification, the only compact irreducible riemannian 
symmetric spaces $G^u/K^u$ with $\rank K < \rank G$ are
{\footnotesize
$$
SU(n)/SO(n), SU(2n)/Sp(n), 
SO(2s+2+2t)/[SO(2s + 1)\times SO(1+2t)], 
E_6/F_4, E_6/Sp(4), (K^u\times K^u)/diag(K^u).
$$
}
The only ones of these spaces 
for which $K^u$ splits are the odd dimensional oriented real
Grassmann manifolds $SO(2s+2+2t)/[SO(2s + 1)\times SO(1+2t)]$ with
$s, t \geqq 1$.  Thus we look at
\begin{equation}\label{odd-fibr-ps}
\begin{aligned}
\pi: \widetilde{M} \to M &\text{ given by } G/K_1 \to G/K_1K_2 \text{ where }
     G = SO^0(2s+1,1+2t),\\
& K_1 = SO^0(u,v)\, \text{ and }\, K_2 =  SO^0(a,b)
\end{aligned}
\end{equation}
with conditions $u+a=2s+1$ and $v+b = 1+2t$ for the signature of 
$\R^{2s+1,1+2t}$, and $u+v = 2s+1$ and $a+b = 1+2t$ for $K_1^u = SO(2s+1)$
and $K_2^u = SO(1+2t)$.  Here $a$ determines $u$, $v$ and $b$, so in fact
$$
K_1 = SO^0(2s+1-a,a) \text{ and } K_2 = SO^0(a,1+2t-a) \text{ for }
	0 \leqq a \leqq \min(2s+1,1+2t).
$$

The symmetry of the pseudo--riemannian symmetric space $M$ is $\rho =
\Ad\left ( \begin{smallmatrix} I_{2s+1} & 0\\ 0 & -I_{1+2t} \end{smallmatrix} 
\right )$.
If $a$ is even then $K_1$ has outer automorphism $\sigma_{1,a} =
\Ad\left ( \begin{smallmatrix} I_{2s} & 0\\ 0 & -1 \end{smallmatrix} \right )$
and $K_2$ has outer automorphism $\sigma_{2,a} =
\Ad\left ( \begin{smallmatrix} -1 & 0\\ 0 & I_{2t} \end{smallmatrix} \right )$;
then $\rho$ and $\sigma_{2,a}$ belong to the same component of $O(2s+1,1+2t)$.
If $a$ is odd then $K_1$ has outer automorphism $\tau_{1,a} =
\Ad\left ( \begin{smallmatrix} -1 & 0\\ 0 & I_{2s} \end{smallmatrix} \right )$ 
and $K_2$ has outer automorphism $\tau_{2,a} =
\Ad\left ( \begin{smallmatrix} I_{2t} & 0\\ 0 & -1 \end{smallmatrix} \right )$;
then $\rho$ and $\tau_{1,a}$ belong to the same component of $O(2s+1,1+2t)$.
Combining this with Corollary \ref{fiber-inv-ps}, we have

\begin{proposition}\label{iso-st-ps}
The full isometry group $I(\widetilde{M}) =
\left ( O(2s+1,1+2t) \times r(SO(a,1+2t-a)) \right )$.  
\end{proposition}

The proof of Corollary \ref{homog-q-ps} is valid for our isotropy--split
fibrations (\ref{odd-fibr-ps}), except that we reduce to Theorem \ref{cw-st} 
instead of Theorem \ref{cw}.  Thus:

\begin{corollary}\label{homog-q-psst}
Let $\pi: \widetilde{M} \to M$ be one of the isotropy--split fibrations 
{\rm (\ref{odd-fibr-ps})}.  Let $p$
denote the projection $\widetilde{G}^\dagger \to G^\dagger$ from
$I(\widetilde{M})$ to $O(2s+1,1+2t)$.

{\rm (1)} If $\gamma \in I(\widetilde{M})$ such that $p(\gamma)$ is
elliptic and the centralizer of $\gamma$ is transitive on $\widetilde{M}$.
Then $p(\gamma) = \pm I$.

{\rm (2)} Consider a pseudo--riemannian covering
$\widetilde{M} \to \Gamma\backslash \widetilde{M}$ such that $p(\Gamma)$
has compact closure $($for example such that $p(\Gamma)$ is
finite or $p(\Gamma) \subset K)$.  Then $\Gamma\backslash \widetilde{M}$ 
is homogeneous if and only if $p(\Gamma) \subset \{\pm I\}$.
\end{corollary}

%
%
%
%

\medskip
\noindent Department of Mathematics \hfill\newline
\noindent University of California\hfill\newline
\noindent Berkeley, California 94720-3840, USA\hfill\newline
\smallskip
\noindent {\tt jawolf@math.berkeley.edu}

\enddocument
\end